\documentclass[a4paper, oneside]{amsart}
\usepackage{amsmath,amsthm,amsfonts,amssymb}
\usepackage{mathtools}
\usepackage[pdftex,bookmarks=true]{hyperref} 
\usepackage{graphicx}
\usepackage{enumitem}
\usepackage{tikz-cd}
\usepackage[a4paper, textheight=22cm, textwidth=16cm]{geometry}

\newcommand{\R}{\mathbb{R}}

\DeclareMathOperator{\GL}{GL}

\numberwithin{equation}{section}

\newtheoremstyle{referencedlemma}
  {12pt plus 4pt minus 2pt}   
  {12pt plus 4pt minus 2pt}   
  {\itshape}  
  {0pt}       
  {\bfseries} 
  {.}         
  {5pt plus 1pt minus 1pt} 
  {\thmname{#1}\thmnumber{ #2}\thmnote{ #3}} 

\theoremstyle{referencedlemma}
\newtheorem{lemref}[equation]{Lemma}

\theoremstyle{plain} 
\newtheorem{thm}{Theorem}
\newtheorem{cor}[thm]{Corollary}
\newtheorem{lem}[equation]{Lemma}

\newtheorem{theorem}{Theorem}
\newtheorem*{theorem*}{Theorem}

\theoremstyle{remark}

\newtheorem*{rem*}{Remark}

\title{Higher path groupoids and the holonomy of formal power series connections}
\author{Matthew Cellot}

\begin{document}

\begin{abstract}
\noindent Building on ideas of Kohno, we develop a framework for the construction of higher holonomy functors via the transport of formal power series connections. Using these techniques, we obtain functors from the path groupoid, the path $2$-groupoid, and the path $3$-groupoid of a manifold. We show that the resulting higher holonomy functors can detect non-trivial classes in higher homotopy groups. As an application, we construct a Gray functor from the path $3$-groupoid of the configuration space of $m$ points in $\mathbb R^n$ for~$n\geqslant 4$ and compute its restriction to $3$-tracks with constant boundary when $n=4$.
\end{abstract}

\maketitle

\section*{Introduction}
\noindent The purpose of this article is to construct higher categorical analogues of holonomy functors. Let us first recall the construction of holonomy functors. Let~$M$ be a smooth manifold and consider a connection $\nabla$ on a vector bundle $E\to M$. Given a smooth curve $\gamma\colon [0,1]\to M$, the \emph{parallel transport} for $\nabla$ defines a way of moving elements of the fibers along $\gamma$, that is, a linear map
$$
\Theta(\gamma)\colon E_{\gamma(0)}\,\tilde{\to}\,E_{\gamma(1)},
$$
where $E_x$ is the fiber of $E$ above a point $x\in M$.  In fact, the points and the smooth paths of~$M$ assemble to form a category $\mathcal{P}_1(M)$ called the \emph{path groupoid} of $M$ (see Section \ref{higherpathgroupoids}), and the parallel transport for~$\nabla$ defines a \emph{holonomy functor} \[ \mathrm{Hol}^1_\nabla\colon \mathcal{P}_1(M)\longrightarrow\mathrm{Vect} \] that sends a point $x\in M$ to the vector space $E_x$, and a smooth path $\gamma$ from $x$ to $y$ to the linear map 
$$
\Theta(\gamma)\colon E_x\,\tilde{\to}\,E_y. 
$$
In the particular case of a trivial vector bundle $E=M\times V$, a connection on $E$ can equivalently be described as an~\mbox{$\mathrm{End}(V)$-valued} differential $1$-form
$$
\omega\in C^1_{dR}(M)\otimes\mathrm{End}(V),
$$
and parallel transport along a smooth curve $\gamma\colon[0,1]\to M$ induces the linear map
$$
\Theta(\gamma) = \mathrm{Id}_V+\int_\gamma\omega+\int_\gamma\omega\omega +\int_\gamma\omega\omega\omega+\cdots \colon V\,\tilde{\to}\,V, 
$$
where $\displaystyle \int\omega\ldots\omega$ denotes the iterated integral of $\omega,\ldots,\omega$, as defined by Chen \cite{chen1973} (see Section \ref{iteratedintegrals}). This defines a holonomy functor
$$
\mathrm{Hol}^1_\omega\colon\mathcal{P}_1(M)\longrightarrow\GL(V), 
$$
where $\GL(V)$ is the category with a single object and with morphisms given by the automorphisms of~$V$. Several interesting functors arise in this setting. If we denote by $\mathrm{Conf}(m,\mathbb{C})$ the configuration space of~$m$ ordered pairwise distinct points in $\mathbb{C}$, then a loop
$$
\gamma\colon[0,1]\longrightarrow\mathrm{Conf}(m,\mathbb C)
$$
can be seen as a pure braid with $m$ strands. Moreover, the homotopy class of a loop in $\mathrm{Conf}(m,\mathbb{C})$ corresponds to the isotopy class of a pure braid, and we obtain an isomorphism between the fundamental group of $\mathrm{Conf}(m,\mathbb C)$ and the pure braid group $P_m$. The Knizhnik--Zamolodchikov connection $\omega_{KZ}$ is a flat connection on a trivial vector bundle
$$
\mathrm{Conf}(m,\mathbb C)\times V\longrightarrow \mathrm{Conf}(m,\mathbb C), 
$$
where the vector space $V$ is obtained by taking the tensor product of $m$ representations of a simple Lie algebra. Parallel transport for $\omega_{KZ}$ induces representations of $P_m$ and of the braid group. The Drinfeld--Kohno theorem identifies the braid-group representations induced by the Knizhnik--Zamolodchikov connection with those obtained from the corresponding quantum group, see \cite[Chapter~5]{ohtsuki2002quantum}. These representations are also related to the Kontsevich integral, a universal finite-type invariant of knots, see \cite[Chapter~6]{ohtsuki2002quantum}.

In order to construct higher categorical analogues of holonomy functors, we build on a method due to Kohno \cite{kohno2016higher,kohno2021higher,Koh2022}. We first introduce higher categorical versions of the path groupoid of a smooth manifold $M$: the path $2$-groupoid $\mathcal{P}_2(M)$ and the path $3$-groupoid $\mathcal{P}_3(M)$. These structures allow us to consider higher homotopical data of $M$. Let $A$ be an augmented differential graded algebra. We associate to $A$ a strict $2$-category $\mathcal{C}_2(A)$ and a Gray $3$-category $\mathcal{C}_3(A)$. We then show that the transport of a formal power series connection with values in $A$ yields functors from the path $2$-groupoid and the path $3$-groupoid of $M$ to $\mathcal{C}_2(A)$ and $\mathcal{C}_3(A)$, respectively. 
\begin{theorem}\label{thm:main}Let $A$ be an augmented differential graded algebra and let
$$
\omega\in C^*_{dR}(M)\,\widehat{\otimes}\,A
$$
be a formal power series connection on a smooth manifold $M$. The transport of $\omega$ induces a $2$-functor and a Gray functor
$$
\mathrm{Hol}_\omega^2\colon \mathcal{P}_2(M)\longrightarrow\mathcal{C}_2(A) \qquad\text{and}\qquad \mathrm{Hol}_\omega^3\colon \mathcal{P}_3(M)\longrightarrow\mathcal{C}_3(A). 
$$
\end{theorem} 
\noindent Theorem \ref{thm:main} corresponds to Theorems \ref{thm2holonomy} and \ref{thm3holonomy} in Section \ref{higherholonomyfunctors}.  

A distinction between two notions of higher thin homotopy plays an important role in these constructions. If two $2$-paths are \emph{rank-$2$ homotopic}, the integrals of the transport over these paths need not be equal. Their difference is, however, contained in a subspace of boundaries of products of degree-$1$ elements. This makes it possible to define the holonomy $2$-functor after passing to a suitable quotient in the target $2$-category. On the other hand, the integration of the transport is strictly invariant under the stronger relation of \emph{laminated rank-$2$ homotopy}. This strict invariance is used in the construction of the holonomy $3$-functor. 

We illustrate the non-triviality of these higher holonomy functors by considering a formal power series connection on the sphere $S^2$.  Its holonomy $2$-functor detects generators of $\pi_2(S^2)$, and its holonomy $3$-functor detects a non-trivial class in $\pi_3(S^2)$.

We apply our construction to a formal power series connection introduced by Komendarczyk, Koytcheff and Voli\'c \cite{KKV}, which generalizes the Knizhnik--Zamolodchikov connection. This formal power series connection is defined on the configuration space $\mathrm{Conf}(m,\mathbb R^n)$ for $n\geqslant4$ and takes values in the cobar complex on the graded dual $\bar{\mathcal D}(m)^*$ of Kontsevich's differential graded algebra of oriented diagrams. We obtain a Gray functor
$$
\mathrm{Hol}^3_\omega\colon \mathcal{P}_3\bigl(\mathrm{Conf}(m,\mathbb R^n)\bigr) \longrightarrow \mathcal{C}_3\bigl(\Omega\bar{\mathcal D}(m)^*\bigr),
$$
where $\Omega$ is the cobar construction. For degree reasons, this functor is trivial on $1$-tracks and laminated $2$-tracks, and it is also trivial on $3$-tracks when $n\geqslant5$. In the remaining case $n=4$, its restriction to $3$-tracks with constant boundary induces an isomorphism
$$
\pi_3\bigl(\mathrm{Conf}(m,\mathbb R^4)\bigr)\otimes\mathbb R \overset{\cong}{\longrightarrow} \mathbb R\{[\Gamma_{ij}^*]\mid 1\leqslant i<j\leqslant m\},
$$
where $\Gamma_{ij}$ is the diagram consisting of a single edge joining the external vertices $i$ and $j$.  This is explained in Theorem \ref{thm:holonomy-3functor}. 

We expect similar techniques to define holonomy $n$-functors for $n\geqslant 4$, but these constructions are currently hindered by the author's ignorance of a good notion of a path $n$-groupoid for $n\geqslant 4$.  

This work is motivated by the study of higher-order versions of the Knizhnik--Zamolodchikov equation and by the use of iterated integrals to study higher homotopy. Chen's theory of formal power series connections relates the transport of such a connection to the homology of based loop spaces \cite{chen1973,chen1977}. Related techniques appear in the work of Cohen and Gitler \cite{cohen2002loop}, who give detailed descriptions of the homology of loop spaces of configuration spaces. We also refer to Baez and Schreiber \cite{baez2004higher} for higher gauge theory and $2$-connections, and to Faria Martins and Picken \cite{martins2010two,faria_martins_fundamental_2011} for higher-dimensional holonomy and fundamental Gray $3$-groupoids. Further work on higher holonomy includes that of Kohno \cite{kohno2016higher,kohno2020higher,kohno2021higher} and Schreiber and Waldorf \cite{schreiber_parallel_2007,schreiber_smooth_2008}. 

The paper is organized as follows. In Section \ref{higherpathgroupoids}, we recall the definitions of the path groupoid, the path $2$-groupoid, and the path $3$-groupoid of a smooth manifold. In Section \ref{iteratedintegrals}, we recall Chen's notion of iterated integral \cite{chen1973}. In Section \ref{formalpowerseriesconnections}, we define formal power series connections and their transport, and we recall their relationship with the homology of based loop spaces. In Section \ref{higherholonomyfunctors}, we construct the holonomy functor, the holonomy $2$-functor, and the holonomy $3$-functor of a formal power series connection, and we prove Theorem \ref{thm:main}. We also compute the higher holonomy functors associated with a formal power series connection on $S^2$. Finally, in Section \ref{application}, we describe the holonomy $3$-functor associated with the formal power series connection of Komendarczyk, Koytcheff and Voli\'c.

Throughout this work, all vector spaces, algebras, coalgebras, and homology groups are taken over the field $\R$.

\subsection*{Acknowledgements}I am extremely grateful to Toshitake Kohno for his guidance during my stay at the University of Tokyo, and for helping me discover the world of higher holonomy functors.

\subsection*{Note on AI use}All of the mathematical ideas, arguments, and computations appearing in this paper are due to the author.  LLMs were only used in the final stage of preparing this manuscript to check spelling and correct minor mathematical inaccuracies. 

\section{Higher path groupoids}\label{higherpathgroupoids}

\noindent In this section, we follow the presentation given by Faria Martins and Picken \cite{faria_martins_fundamental_2011} of the path groupoid, the path 2-groupoid, and the path 3-groupoid of a smooth manifold. Similar definitions can be found in \cite{schreiber_parallel_2007, schreiber_smooth_2008, kohno2016higher,kohno2020higher,kohno2021higher, martins2010two}.

Throughout this section, we suppose $M$ is a smooth manifold (possibly with boundary).  

\subsection{Higher order paths}\label{n-paths}
An \emph{$n$-path}, for an integer $n\geqslant 0$, is a smooth map $\alpha\colon [0,1]^n\to M$ for which there exists an $\varepsilon > 0$ such that for all $i\in \{1, \ldots n\}$,
\begin{align*}
0\leqslant t_i \leqslant \varepsilon & \Rightarrow \alpha(t_1, \ldots , t_i, \ldots , t_n) = \alpha(t_1, \ldots,  0, \ldots , t_n),\\
1-\varepsilon \leqslant t_i \leqslant 1 &\Rightarrow \alpha(t_1, \ldots , t_i, \ldots, t_n) = \alpha(t_1, \ldots , 1, \ldots t_n), 
\end{align*}
and $\alpha([0,1]^{n-1}\times \{0\})$ and $\alpha([0,1]^{n-1}\times \{1\})$ both consist of a single point. In particular, a $0$-path is just a point.

An \emph{$n$-path from $f$ to $g$}, where $f$ and $g$ are $(n-1)$-paths, is an $n$-path $\alpha\colon [0,1]^{n}\to M$ such that
$$
\alpha(0,-,\cdots, -) = f \quad \text{and} \quad
\alpha(1, -, \cdots , -) = g.
$$

The \emph{composition of two 1-paths} $\gamma_1$ and $\gamma_2$ such that $\gamma_1(1) = \gamma_2(0)$ is the 1-path $\gamma_1 \circ \gamma_2$ defined by:
\begin{equation*}
\gamma_1\circ  \gamma_2(t) = 
\begin{cases}
\gamma_1(2t) & \text{if} \ 0 \leqslant t \leqslant \frac{1}{2},\\
\gamma_2(2t-1) & \text{if}\  \frac{1}{2}\leqslant t \leqslant 1.
\end{cases}
\end{equation*}

\subsection{Rank-1 homotopies}
A \emph{rank-1 homotopy} is a 2-path $h\colon [0,1]^2\to M$ such that for all $(s,t) \in [0,1]^2$,  the rank of its differential $d_{(s, t)}h$ is less than 2.  Two 1-paths $\gamma$ and $\gamma'$ are \emph{rank-1 homotopic} if there exists a rank-1 homotopy from $\gamma$ to $\gamma'$.

A \emph{$1$-track} $[\gamma]$ is the rank-1 homotopy class of a 1-path $\gamma$.  The \emph{composition $[\gamma_1]\circ [\gamma_2]$ of two 1-tracks}~$[\gamma_1]$ and $[\gamma_2]$ such that $\gamma_1(1) = \gamma_2(0)$ is the 1-track $[\gamma_1\circ \gamma_2]$.  Notice that the composition of two 1-tracks does not depend on the choice of representative 1-paths and that it is associative. Moreover, if we denote by~$c_{x}$ the constant 1-path at a point $x\in M$, then for any 1-path $\gamma$,
$$
[\gamma] \circ [c_{\gamma(1)}] = [\gamma] \quad \text{and} \quad [c_{\gamma(0)}] \circ [\gamma] = [\gamma].
$$

\subsection{Path groupoids}
The \emph{path groupoid} of $M$ is the category $\mathcal{P}_1(M)$ whose objects are given by the points of $M$ and whose morphisms are 1-tracks, with composition induced by the composition of 1-tracks.  Notice that every 1-track is invertible. 

\subsection{Vertical composition of 2-paths}\label{vertical composition}Let $g$ and $h$ be 2-paths such that $g(1, -) = h(0,-)$. The \emph{vertical composition} of $g$ and $h$ is the 2-path $g\cdot h$ defined by:
\begin{equation*}
(g\cdot h)(s,t) = 
\begin{cases}
g(2s,t) &\text{if}\ 0\leqslant s \leqslant \frac{1}{2} \ \text{and} \ 0\leqslant t \leqslant 1,\\
h(2s-1,t) &\text{if}\ \frac{1}{2}\leqslant s \leqslant 1 \ \text{and} \ 0\leqslant t \leqslant 1.
\end{cases}
\end{equation*}

\subsection{Horizontal composition of 2-paths}

Let $h_1$ and $h_2$ be 2-paths such that $$h_1([0,1]\times \{1\}) = h_2([0,1]\times \{0\}).$$ The \emph{horizontal composition} of $h_1$ and $h_2$ is the 2-path $h_1\circ h_2$ defined by:
\begin{equation*}
(h_1\circ h_2) (s,t) = 
\begin{cases}
h_1(s,2t) &\text{if}\ 0\leqslant s \leqslant 1 \ \text{and} \ 0\leqslant t \leqslant \frac{1}{2},\\
h_2(s,2t-1) &\text{if}\ 0\leqslant s \leqslant 1 \ \text{and} \ \frac{1}{2}\leqslant t \leqslant 1.
\end{cases}
\end{equation*}

\subsection{Rank-2 homotopies} A \emph{rank-2 homotopy} is a 3-path $\alpha\colon [0,1]^3\to M$ such that:
\begin{itemize}
\item for all $(r,s,t) \in [0,1]^3$, the rank of its differential $d_{(r,s,t)}\alpha$ is less than 3,
\item $\alpha(-,0,-)$ and $\alpha(-,1,-)$ are rank-1 homotopies.
\end{itemize}
Two 2-paths $g$ and $h$ are \emph{rank-2 homotopic} if there exists a rank-2 homotopy from $g$ to $h$.

A \emph{2-track} $[g]$ is the rank-2 homotopy class of a 2-path $g$. The \emph{vertical composition of~2-tracks} is induced by the vertical composition of 2-paths, and the \emph{horizontal composition of 2-tracks} is induced by the horizontal composition of 2-paths. Notice that these compositions are associative and satisfy the interchange law. 

\subsection{Path 2-groupoids} The \emph{path 2-groupoid of $M$} is the strict 2-category $\mathcal{P}_2 (M)$ defined as follows.  The objects of $\mathcal{P}_2 (M)$ are the points of $M$. The 1-morphisms of $\mathcal{P}_2 (M)$ are given by 1-tracks, with composition induced by the composition of 1-tracks.  A 2-morphism in $\mathcal{P}_2 (M)$ from $[\gamma]$ to $[\gamma']$ is given by a 2-track $[g]$ such that $[g(0, -)] = [\gamma]$ and $[g(1,-)] = [\gamma']$.  The identity of a 1-morphism $[\gamma]$ is given by the rank-2 homotopy class of the constant 2-path at $\gamma$.  The vertical and the horizontal composition of 2-morphisms are induced by the vertical and the horizontal composition of~2-paths, respectively. We refer the reader to \cite{martins2010two} for the details on the vertical and the horizontal composition of 2-morphisms. Notice that all 2-tracks are invertible.

\subsection{Laminated rank-2 homotopies}
We recall a stronger notion of rank-2 homotopy due to Faria Martins and Picken \cite{faria_martins_fundamental_2011}. A \emph{laminated rank-2 homotopy} is a rank-2 homotopy $\alpha$ such that for all~$(r,s) \in [0,1]^2$ at least one of the following conditions holds:
\begin{itemize}
\item for all $t\in [0,1]$, there exist constants $(a_t,b_t)\neq (0,0)$ such that 
$$
a_t\frac{\partial}{\partial r}\alpha(r,s,t) + b_t\frac{\partial}{\partial t}\alpha(r,s,t) = 0,
$$
\item for all $t\in [0,1]$, there exist constants $(a_t,b_t)\neq (0,0)$ such that 
$$
a_t\frac{\partial}{\partial s}\alpha(r,s,t) + b_t\frac{\partial}{\partial t}\alpha(r,s,t) = 0,
$$
\item there exist constants $(a,b) \neq (0,0)$ such that, for all $t\in [0,1]$,
\begin{equation*}
a\frac{\partial}{\partial r}\alpha(r,s,t) + b\frac{\partial}{\partial s}\alpha(r,s,t) = 0.
\end{equation*}
\end{itemize}
A \emph{laminated 2-track} $[g]$ is the laminated rank-2 homotopy class of a 2-path $g$. 

\subsection{Whiskering of 2-paths by 1-paths}
Let $\gamma$ be a 1-path and $g$ be a 2-path such that $$\gamma(1) = g([0,1]\times \{0\}).$$
The \emph{left whiskering of $g$ with $\gamma$} is the 2-path $\gamma\circ g$ given by:
\begin{equation*}
(\gamma\circ g) (s,t) = 
\begin{cases}
\gamma(2t) & \text{if } 0\leqslant t \leqslant \frac{1}{2} \text{ and } 0\leqslant s \leqslant 1,\\
g(s, 2t-1) & \text{if } \frac{1}{2}\leqslant t \leqslant 1 \text{ and } 0\leqslant s \leqslant 1.
\end{cases}
\end{equation*}

Let $\gamma$ be a 1-path and $g$ be a 2-path such that $$g([0,1]\times \{1\}) = \gamma(0).$$
The \emph{right whiskering of $g$ with $\gamma$} is the 2-path $g\circ \gamma$ given by:
\begin{equation*}
(g\circ \gamma) (s,t) = 
\begin{cases}
g(s, 2t) & \text{if } 0\leqslant t \leqslant \frac{1}{2} \text{ and } 0\leqslant s \leqslant 1,\\
\gamma(2t-1) & \text{if } \frac{1}{2}\leqslant t \leqslant 1 \text{ and } 0\leqslant s \leqslant 1.
\end{cases}
\end{equation*}
The \emph{left (respectively right) whiskering of a laminated 2-track with a 1-track} is the laminated rank-2 homotopy class of the left (respectively right) whiskering of a representative 2-path with a representative 1-path. Notice that this operation does not depend on the choice of representatives.

\subsection{Horizontal compositions of laminated 2-tracks}

The horizontal composition of 2-paths does not descend to the quotient by laminated rank-2 homotopy.  Therefore, we use whiskering to define two more notions of horizontal composition. 

Let $h_1$ and $h_2$ be 2-paths such that $$h_1([0,1]\times \{1\}) = h_2([0,1]\times \{0\}).$$ 
We denote by 
$\begin{pmatrix}
& h_2\\
h_1 & 
\end{pmatrix}$ the 2-path defined by $\begin{pmatrix}
& h_2\\
h_1 & 
\end{pmatrix} = (h_1\circ h_2(0,-))\cdot (h_1(1,-) \circ h_2)$. More explicitly:
\begin{equation*}
\begin{pmatrix}
& h_2\\
h_1 & 
\end{pmatrix} (s,t) = 
\begin{cases}
h_1(2s, 2t) & \text{if } 0\leqslant s \leqslant \frac{1}{2} \text{ and } 0\leqslant t \leqslant \frac{1}{2},\\
h_1(1, 2t) & \text{if } \frac{1}{2}\leqslant s \leqslant 1 \text{ and } 0\leqslant t \leqslant \frac{1}{2},\\
h_2(0, 2t-1) & \text{if } 0 \leqslant s \leqslant \frac{1}{2} \text{ and } \frac{1}{2}\leqslant t \leqslant 1,\\
h_2(2s-1, 2t-1) & \text{if } \frac{1}{2} \leqslant s \leqslant 1 \text{ and } \frac{1}{2}\leqslant t \leqslant 1.\\
\end{cases}
\end{equation*}
We denote by $\begin{pmatrix}
h_1 & \\
& h_2
\end{pmatrix}$ the 2-path defined by $\begin{pmatrix}
h_1 & \\
& h_2
\end{pmatrix} = (h_1(0,-)\circ h_2)\cdot (h_1\circ h_2(1,-))$. More explicitly:
\begin{equation*}
\begin{pmatrix}
h_1& \\
 & h_2
\end{pmatrix} (s,t) = 
\begin{cases}
h_1(0, 2t) & \text{if } 0\leqslant s \leqslant \frac{1}{2} \text{ and } 0\leqslant t \leqslant \frac{1}{2},\\
h_1(2s-1, 2t) & \text{if } \frac{1}{2}\leqslant s \leqslant 1 \text{ and } 0\leqslant t \leqslant \frac{1}{2},\\
h_2(2s, 2t-1) & \text{if } 0 \leqslant s \leqslant \frac{1}{2} \text{ and } \frac{1}{2}\leqslant t \leqslant 1,\\
h_2(1, 2t-1) & \text{if } \frac{1}{2} \leqslant s \leqslant 1 \text{ and } \frac{1}{2}\leqslant t \leqslant 1.\\
\end{cases}
\end{equation*}
These compositions induce compositions of laminated 2-tracks denoted by $\begin{pmatrix}
& [h_2]\\
[h_1] & 
\end{pmatrix}$ and
$\begin{pmatrix}
[h_1] & \\
& [h_2]
\end{pmatrix}$, respectively. 

\subsection{Upward composition of 3-paths}
Let $J$ and $K$ be 3-paths such that 
$$J(1,-,-) = K(0,-,-).$$  The \emph{upward composition} of $J$ and $K$ is the 3-path $J*K$ defined by:
\begin{equation*}
(J*K) (r,s,t) = 
\begin{cases}
J(2r,s,t) &\text{if } 0\leqslant r \leqslant \frac{1}{2}, \ 0\leqslant s \leqslant 1 \text{ and }0\leqslant t \leqslant 1,\\
K(2r-1, s,t) & \text{if }\frac{1}{2}\leqslant r \leqslant 1, \ 0\leqslant s \leqslant 1 \text{ and }0\leqslant t \leqslant 1.
\end{cases}
\end{equation*}

\subsection{Vertical composition of 3-paths}
Let $J$ and $K$ be 3-paths such that 
$$J(-,1,-) = K(-,0,-).$$ 
The \emph{vertical composition} of $J$ and $K$ is the 3-path $J\cdot K$ defined by:
\begin{equation*}
(J\cdot K) (r,s,t) = 
\begin{cases}
J(r,2s,t) &\text{if } 0\leqslant r \leqslant 1, \ 0\leqslant s \leqslant \frac{1}{2} \text{ and }0\leqslant t \leqslant 1,\\
K(r, 2s-1,t) & \text{if }0\leqslant r \leqslant 1, \ \frac{1}{2}\leqslant s \leqslant 1 \text{ and }0\leqslant t \leqslant 1.
\end{cases}
\end{equation*}

\subsection{Rank-3 homotopies}\label{Subsection_Rank-3 homotopies}
A \emph{good 3-path} is a 3-path $J\colon [0,1]^3 \to M$ such that 
\begin{itemize}
\item  for all $r, r'\in [0,1]$ and $t\in [0,1]$, $J(r,0,t) = J(r', 0,t)$,
\item for all $r, r'\in [0,1]$ and $t\in [0,1]$, $J(r,1,t) = J(r', 1,t)$.
\end{itemize} 
A \emph{rank-3 homotopy} is a 4-path $W \colon [0,1]^4\to M$ such that
\begin{itemize}
\item for all $(q,r,s,t) \in [0,1]^4$, the rank of $d_{(q,r,s,t)}W$ is less than 4,
\item the restrictions $W(-,0,-,-)$ and $W(-,1,-,-)$ are laminated rank-2 homotopies,
\item the restrictions $W(-,r,0,-)$ and $W(-,r,1,-)$ are independent of $r\in [0,1]$ and define rank-1 homotopies.
\end{itemize}
Two good 3-paths $J$ and $J'$ are \emph{rank-3 homotopic} if there exists a rank-3 homotopy from $J$ to $J'$.
A \emph{3-track}~$[J]$ is the rank-3 homotopy class of a good 3-path $J$.

\subsection{Whiskering of 3-tracks by 1-tracks} Let $\gamma$ be a 1-path and let $J$ be a 3-path such that $$\gamma(1) = J([0,1]^2\times \{0\}).$$ The \emph{left whiskering of $J$ with $\gamma$} is the 3-path $\gamma\circ J$ defined by:
\begin{equation*}
(\gamma\circ J)(r,s,t) = 
\begin{cases}
\gamma(2t) & \text{if } 0\leqslant r \leqslant 1, \ 0\leqslant s\leqslant 1 \text{ and } 0\leqslant t \leqslant \frac{1}{2},\\
J(r,s,2t-1) & \text{if }0\leqslant r \leqslant 1, \ 0\leqslant s\leqslant 1 \text{ and } \frac{1}{2}\leqslant t \leqslant 1.
\end{cases}
\end{equation*}

Let $\gamma$ be a 1-path and let $J$ be a 3-path such that $$J([0,1]^2\times \{1\}) = \gamma(0).$$
The \emph{right whiskering of $J$ with $\gamma$} is the 3-path $J\circ \gamma$ defined by:
\begin{equation*}
(J\circ \gamma) (r,s,t) = 
\begin{cases}
J(r,s,2t) & \text{if } 0\leqslant r \leqslant 1, \ 0\leqslant s\leqslant 1 \text{ and } 0\leqslant t \leqslant \frac{1}{2},\\
\gamma(2t-1) & \text{if }0\leqslant r \leqslant 1, \ 0\leqslant s\leqslant 1 \text{ and } \frac{1}{2}\leqslant t \leqslant 1.
\end{cases}
\end{equation*}

The \emph{left (respectively right) whiskering of a 3-track with a 1-track} is the rank-3 homotopy class of the left (respectively right) whiskering of a representative good 3-path with a representative 1-path. Notice that this operation does not depend on the choice of representatives.

\subsection{Path 3-groupoids}

The \emph{path 3-groupoid of $M$} is the Gray 3-groupoid $\mathcal{P}_3(M)$ defined as follows. The objects of $\mathcal{P}_3(M)$ are the points of $M$.  The \mbox{1-morphisms} of $\mathcal{P}_3(M)$ are given by 1-tracks, with composition induced by the composition of 1-tracks.  A 2-morphism in $\mathcal{P}_3(M)$ from $[\gamma]$ to $[\gamma']$ is given by a laminated 2-track $[g]$ such that $[g(0,-)] = [\gamma]$ and $[g(1,-)] = [\gamma']$. The vertical composition and the two horizontal compositions of 2-morphisms are induced by the vertical composition and the two horizontal compositions of 2-paths, respectively. A 3-mor\-phism in $\mathcal{P}_3(M)$ from $[g]$ to $[h]$ is given by a 3-track $[J]$ such that $[J(0,-,-)] = [g]$ and $[J(1,-,-)] = [h]$. The upward composition, the vertical composition and the two horizontal compositions of 3-morphisms are induced by the upward composition, the vertical composition and the two horizontal compositions of 3-paths, respectively.  We refer the reader to \cite{faria_martins_fundamental_2011} for more details on $\mathcal{P}_3(M)$ and for the proof that it is a Gray 3-groupoid. 

\section{Iterated integrals}\label{iteratedintegrals}
\noindent In this section, we recall the notion of a differential form on a differentiable space,  and we describe a certain class of differential forms on path spaces called iterated integrals.  We refer the reader to \cite{chen1973, chen1977} for more details.

\subsection{Differentiable spaces and differentiable maps}
A \emph{convex $n$-region} (or simply a \emph{convex region}) is a closed convex subset of~$\R^n$, for some $n \geqslant 0$.  A map $f$ from a convex $n$-region $U$ to a convex region~$V$ is \emph{smooth} if there exists a smooth extension of $f$ to an open subset of $\mathbb{R}^n$.  A \emph{differentiable space} is a Hausdorff space $X$ equipped with a family of maps, called \emph{plots}, satisfying the following conditions:
\begin{itemize}
\item every plot is a continuous map $\phi : U \to X$, where $U$ is a convex region,
\item if $\phi : U \to X$ is a plot,  $U'$ is a convex region, and $f : U' \to U$ is a smooth map, then $\phi \circ f$ is a plot,
\item every map $\{0\} \to X$ is a plot.
\end{itemize}
We say that a plot $\psi$ with domain $V$ \emph{goes through} another plot $\phi$ with domain $U$ via $f$ if there exists a smooth map $f : V \to U$ such that $\psi = \phi\circ f$. Let $X$ and $X'$ be differentiable spaces. A \emph{differentiable map} from $X$ to $X'$ is a continuous map $f : X \to X'$ such that, for every plot $\phi$ of $X$, $f\circ \phi$ is a plot of $X'$.

The notion of differentiable space is weaker than that of smooth manifold and will help us avoid certain limitations associated to smooth manifolds.  In particular, smooth manifolds are differentiable spaces with plots given by smooth maps $\phi : U \to M$ where $U$ is a convex region.

\subsection{Differential forms on differentiable spaces}
For any $p\geqslant 0$, a \emph{differential $p$-form} on a differentiable space $X$ is a rule that assigns to each plot $\phi : U \to X $ a differential $p$-form $\omega_\phi \in C_{dR}^p(U)$ such that, if $\psi$ is a plot that goes through $\phi$ via $f$, then
\begin{equation*}
f^*\omega_\phi = \omega_\psi.
\end{equation*}
A \emph{differential form} on $X$ is a differential $p$-form for some $p$ called its \emph{degree}. The set of differential $p$-forms on a differentiable space $X$ is denoted by $C_{dR}^p(X)$. The differential forms on a differentiable space $X$ define a commutative differential graded algebra $C_{dR}^*(X) = \bigoplus_{p\geqslant 0}C_{dR}^p(X)$ whose sum, exterior product and exterior derivative are given by:
\begin{align*}
(\omega + \omega')_\phi &= \omega_\phi + \omega'_\phi,\\
(\omega \wedge \omega')_\phi &= \omega_\phi \wedge \omega'_\phi,\\
(d\omega)_\phi &= d\omega_\phi,
\end{align*}
for any plot $\phi\colon U \to X$. 

\subsection{Pointed path spaces}
The \emph{pointed path space} of a smooth manifold $M$ is the set $P(M;x_0,x_1)$ of 1-paths in $M$ from~$x_0$ to $x_1$ (see Section \ref{n-paths}) with the compact-open topology, where $x_0$ and $x_1$ are points of~$M$.  The pointed path space of $M$ is a differentiable space when equipped with the following plots. A plot of~$P(M;x_0,x_1)$ is a continuous map $\phi : U \to P(M;x_0,x_1)$, such that the adjoint map~$\widetilde{\phi} : U \times [0,1] \to M$ is smooth.

\subsection{Iterated integrals on path spaces}

Let $\omega_1, \ldots, \omega_k$ be differential forms of positive degree on a smooth manifold $M$.  Let $\phi : U \to P(M;x_0,x_1)$ be a plot of $P(M;x_0,x_1)$,  where $x_0$ and $x_1$ are points of~$M$,  and let $\Delta^k$ be the $k$-simplex given by:
\begin{equation*}
\Delta^k = \{(t_1, \ldots, t_k)\in \mathbb{R}^k \mid 0\leqslant t_1\leqslant \cdots \leqslant t_k\leqslant 1\}.
\end{equation*}
The plot $\phi$ induces a smooth map $\widetilde{\phi} :  \Delta^k\times U \to M^k$ defined by:
\begin{equation*}
\widetilde{\phi}(t_1, \ldots, t_k, x) = \left(\phi(x)(t_1), \ldots,\phi(x)(t_k)\right),
\end{equation*}
for all $((t_1, \ldots, t_k),x)\in \Delta^k\times U$.

We denote by $\pi_i : M^k \to M$ the projection onto the $i$-th factor for any smooth manifold $M$. The \emph{iterated integral} of differential forms $\omega_1, \ldots, \omega_k$ of positive degree on $M$ is the diffe\-rential form~$\displaystyle \int \omega_1 \ldots \omega_k$ on the differen\-tiable space $P(M;x_0,x_1)$ defined by:
\begin{equation*}
\left(\int \omega_1 \ldots \omega_k\right)_{\!\phi} = \int_{\Delta^k}\widetilde{\phi}^*(\pi_1^*\omega_1\wedge \cdots \wedge \pi_k^*\omega_k),
\end{equation*}
for any plot $\phi : U \to P(M;x_0,x_1)$, where $\displaystyle \int_{\Delta^k}$ denotes integration along the fiber of $\Delta^k\times U \to U$. Linear combinations of iterated integrals will also be called iterated integrals.  The iterated integral $$\int \omega_1 \ldots \omega_k$$ is a differential form of degree $\deg(\omega_1) + \cdots + \deg(\omega_k) - k$, where $\mathrm{deg}(\omega_i)$ denotes the degree of $\omega_i$ for \break all $i \in \{1,\ldots , k\}$.

We recall a useful formula:
\begin{lemref}[\cite{kohno2016higher}]\label{lem21}
Let $\omega_1, \ldots \omega_k$ be differential forms of positive degree on a smooth mani\-fold~$M$. On the pointed path space~$P(M; x_0, x_1)$, we have:
\begin{align*}
d \int \omega_1\ldots \omega_k = &\sum\limits_{i=1}^k (-1)^{\nu_{i-1} +1}\int \omega_1\ldots \omega_{i-1}\, d\omega_i\, \omega_{i+1} \ldots \omega_k \\
&+ \sum\limits_{i=1}^{k-1}(-1)^{\nu_i +1}\int \omega_1\ldots \omega_{i-1}(\omega_i\wedge \omega_{i+1}) \omega_{i+2}\ldots \omega_k ,
\end{align*}
where $\nu_i = \deg(\omega_1) +\cdots + \deg(\omega_i) - i$.
\end{lemref}

\subsection{The bar complex}\label{barcomplex}

Let $M$ be a smooth manifold and let $x_0$ and $x_1$ be points of $M$.  The \emph{bar complex of ite\-rated integrals on $M$} (or simply the \emph{bar complex on $M$}) is the sub\-complex~$B^*(M)$ of the complex~$C^*_{dR}(P(M; x_0,x_1))$ of differential forms on the pointed path space of $M$ such that for all $q\geqslant 0$, the vector space $B^q(M)$ is the sum of the images of the multilinear maps 
$$
\int \colon C_{dR}^{p_1}(M)\otimes \cdots \otimes C_{dR}^{p_k}(M) \to C_{dR}^q(P(M; x_0,x_1)) 
$$
defined by the iterated integrals, where the sum is taken over all positive integers $k, p_1, \ldots, p_k$ such that~$p_1+ \cdots + p_k - k = q$.  As a consequence of Lemma \ref{lem21}, we have that $d \left(B^q(M)\right) \subset B^{q+1}(M)$. 

A nonzero element of $B^*(M)$ is \emph{homogeneous of length $k$}, for some $k\geqslant 0$,  if it is the sum of iterated integrals of the form:
$$ \int \omega_1 \cdots \omega_k. $$

\subsection{The integration of iterated integrals over bounded plots}\label{integrationiteratedintegrals}

A \emph{bounded} plot of a differentiable space is a plot with bounded domain.  Let $\omega_1, \ldots , \omega_k$ be differential forms of positive degree on a smooth manifold $M$, and let $\phi$ be a bounded plot of $P(M; x_0, x_1)$ with domain $U$. Then the \emph{integration of~$\displaystyle \int \omega_1\ldots \omega_k$ over $\phi$}, denoted by~$\displaystyle \langle \int \omega_1\ldots \omega_k, \phi \rangle$, is defined by the following equation:
\begin{equation*}
\langle \int \omega_1\ldots \omega_k, \phi \rangle =
\begin{cases}\displaystyle
 \int_U \left(\int \omega_1\ldots \omega_k\right)_{\! \phi} & \text{if $\mathrm{deg}(\omega_1) + \cdots + \mathrm{deg}(\omega_k) - k = \dim U$,}\\
 0 & \text{otherwise,}
\end{cases}
\end{equation*}
where $\dim U$ is the affine dimension of $U$.  Note that for $k = 1$, we get the integral:
\begin{equation*}
\langle \int \omega_1, \phi \rangle = \int_{[0,1]\times U}\widetilde{\phi}^*\omega_1,
\end{equation*}
where $\widetilde{\phi}$ is the adjoint map to the plot $\phi$.  In particular, if $U = *$, then the above is simply the integration of a differential form on $M$ along a path. 

For $k = 0$,  we set:
\begin{equation*}
\langle \int \omega_1\ldots \omega_k, \phi \rangle = \delta_0^{\dim U}.
\end{equation*}

For any bounded plot $\phi \colon U \to P(M; x_0,x_1)$, taking the integration over $\phi$ defines a linear map 
$$
\langle\,  \cdot\, , \phi\rangle \colon B^*(M) \to \R.
$$

\section{Formal power series connections}\label{formalpowerseriesconnections}

\noindent In this section, we build on Chen's theory of formal power series connections \cite{chen1973, chen1977} and its formulation by Kohno~\cite{kohno2016higher, kohno2020higher, kohno2021higher, Koh2022}.  We recall the definitions of differential graded algebras and coalgebras, augmentations and coaugmentations, and the cobar construction.  We use these to introduce formal power series connections and their transport in a slightly more general framework than that of Kohno.  We then recall a result relating the transport of a formal power series connection to the homology of the loop space of a simply connected manifold. We refer the reader to \cite{felixRationalHomotopyTheory2001} for more details on graded algebras.  

Recall that all vectors spaces, algebras, and coalgebras are over the field $\R$. 

\subsection{Graded algebras}
A \emph{graded vector space} is a vector space $V$ together with a direct sum decomposition
$$
V = \bigoplus_{n\geqslant 0} V_n.
$$
An element $v$ of $V$ is \emph{homogeneous of degree $n$} (or simply \emph{of degree $n$}) if $v\in V_n$.  A \emph{linear map of degree~$k$} is a linear map $f\colon V \to W$ between graded vector spaces such that
$$
f(V_n) \subset W_{n+k}
$$
for all $n\geqslant 0$.  The \emph{tensor product} $V\otimes W$ of graded vector spaces is defined by 
$$
(V\otimes W)_n = \bigoplus_{i+j= n} (V_i\otimes W_j)
$$
for all $n\geqslant 0$. 

The \emph{desuspension} of a graded vector space $V$ is the graded vector space~$s^{-1}V$ given by $(s^{-1}V)_n = V_{n+1}$ for all $n\geqslant 0$. The element of $s^{-1}V$ corresponding to $v\in V$ is denoted by $s^{-1}v$.  

A \emph{graded algebra} is a graded vector space $A$ together with an associative linear map $A\otimes A \to A$ of degree $0$ that has a unit $1\in A_0$.  We view $\mathbb{R}$ as a graded algebra concentrated in degree $0$. A \emph{morphism} of graded algebras is a linear map $\varphi \colon A \to B$ of degree $0$ such that
$$
\varphi(xy) = \varphi(x)\varphi(y)
$$
for all $x,y \in A$, and $\varphi(1)  =1$.  An \emph{augmentation} of a graded algebra $A$ is a morphism $\varepsilon \colon A \to \mathbb{R}$ of graded algebras. 

\subsection{The tensor algebra on a graded vector space}
Let $V = \bigoplus_{n\geqslant 0} V_n$ be a graded vector space. The \emph{tensor algebra} on $V$ is the algebra $TV$ defined by
$$
TV = \bigoplus_{q\geqslant 0} T^qV
$$
where $T^0V = \mathbb{R}$ and 
$$
T^qV = \underbrace{V\otimes \cdots \otimes V}_q,
$$
with multiplication given by the tensor product. Note that $q$ is not the degree in $TV$: if $v_1, \ldots, v_q$ are homogeneous elements in $V$, then $v_1\otimes \cdots \otimes v_q \in TV$ has degree $\deg(v_1) + \cdots + \deg(v_q)$.  There is a canonical augmentation
$$
\varepsilon \colon TV \to \mathbb{R}
$$
defined by $\varepsilon(1) = 1$ and $\varepsilon(v) = 0$ for all $v\in V$. 

\subsection{Differential graded algebras}

Let $A$ be a graded algebra.  A \emph{derivation of degree $k$} is a linear map $\delta \colon A \to A$ of degree $k$ such that
$$
\delta (xy) = \delta(x)y + (-1)^{k\lvert x \rvert}x\delta(y)
$$
for all $x,y\in A$ such that $x$ is homogeneous of degree $\lvert x \rvert$.  A \emph{differential graded algebra} (or simply a \emph{dga}) is a graded algebra $A$ equipped with a derivation $\delta \colon A \to A$ of degree $-1$ such that $\delta \circ \delta = 0$.  An \emph{augmented} differential graded algebra is a dga equipped with an augmentation $\varepsilon$ such that $\varepsilon\circ \delta = 0$. 

\subsection{Example}

Let $M$ be a simply connected manifold and let 
$$
H_{*}(M; \R) = \bigoplus_{k\geqslant 0} H_k(M; \R)
$$
be the graded vector space given by the singular homology of $M$ with coefficients in $\R$.  The tensor algebra
$$
T(s^{-1}H_{*}(M; \R))
$$
is a graded algebra, where $s^{-1}$ denotes the desuspension.  Chen equips $T(s^{-1}H_{*}(M; \R))$ with a (non-trivial) derivation
$$
\delta \colon T(s^{-1}H_{*}(M; \R)) \to T(s^{-1}H_{*}(M; \R))
$$
making $(T(s^{-1}H_{*}(M; \R)), \delta)$ a differential graded algebra, see \cite{chen1977}.

\subsection{Filtered algebras}

A \emph{filtered algebra} is  a pair $(A, F_\bullet)$ where $A$ is an algebra and $F_\bullet$ is a descending filtration of two-sided ideals
$$
A = F_0 \supset F_1 \supset F_2 \supset \cdots
$$
such that $F_m \cdot F_n \subset F_{m+n}$ for all $m,n\geqslant 0$.  The \emph{completion} of a filtered algebra $(A, F_\bullet)$ is the algebra
$$
\varprojlim_n A/ F_n = \left\lbrace (x_k)_{k\geqslant 1} \in (A/F_{k})_{k\geqslant 1} \mid x_i = x_j + F_i\ \text{for}\ i \leqslant j \right\rbrace.
$$
For any filtered algebra $(A, F_\bullet)$, the surjections $A \to A/F_n$ induce a morphism $\phi \colon A\to \varprojlim_n A/ F_n$ whose kernel is the two-sided ideal 
$$
\mathrm{Ker}(\phi) = \bigcap_{n\geqslant 1} F_n.
$$ 
A filtered algebra $(A, F_\bullet)$ is \emph{complete} if the morphism
$$
\phi \colon A \to \varprojlim_n A/ F_n
$$
is an isomorphism. 

\subsection{Example: the augmentation completion}\label{augmentation completion}

Let $A$ be an algebra with augmentation $\varepsilon\colon A \to \R$.  For $n\geqslant 1$, let $(\mathrm{Ker}(\varepsilon))^n$ be the two-sided ideal generated by elements of the form
$$
x_1\cdots x_n \in A
$$
for $x_1, \ldots, x_n \in \mathrm{Ker}(\varepsilon)$.  This defines a descending filtration
$$
A \supset \mathrm{Ker}(\varepsilon) \supset (\mathrm{Ker}(\varepsilon))^2 \supset (\mathrm{Ker}(\varepsilon))^3\supset \cdots
$$
The \emph{augmentation completion} of $A$ is the algebra
$$
\widehat{A}^\varepsilon = \varprojlim_n A/(\mathrm{Ker}(\varepsilon))^n.
$$

Let $(A, \delta)$ be an augmented differential graded algebra with augmentation $\varepsilon$. An element~\mbox{$(x_n)_{n\geqslant 1} \in \widehat{A}^\varepsilon$} is \emph{homogeneous of degree $k$} if $x_n \in A/(\mathrm{Ker}(\varepsilon))^n$ is homogeneous of degree $k$ for all~$n\geqslant 1$.  Let $\widehat{A}^\varepsilon_k \subset \widehat{A}^\varepsilon$ be the subspace of homogeneous elements of degree $k$, and let $\hat{\delta}\colon \widehat{A}^\varepsilon \to \widehat{A}^\varepsilon$ be the linear map defined by
$$
\hat{\delta}((x_n)_{n\geqslant 1}) = (\delta(x_n))_{n\geqslant 1}
$$
for all $(x_n)_{n\geqslant 1} \in \widehat{A}^\varepsilon$.  We will often abuse notation and write $\delta$ instead of $\hat{\delta}$. Note that 
$$
\hat{\delta}(\widehat{A}^\varepsilon_k) \subset \widehat{A}^\varepsilon_{k-1}
$$
for all $k\geqslant 1$.

\subsection{Differential graded coalgebras}
A \emph{graded coalgebra} is a graded vector space $C$ together with two linear maps
$$
\Delta \colon C \to C\otimes C, \quad \varepsilon \colon C \to \R
$$
of degree $0$ such that 
$$
(\Delta \otimes \mathrm{id}) \circ \Delta = (\mathrm{id}\otimes \Delta) \circ \Delta\quad \text{and}\quad (\mathrm{id}\otimes \varepsilon) \circ \Delta = (\varepsilon \otimes \mathrm{id}) \circ \Delta = \mathrm{id}.
$$
We view $\R$ as a graded coalgebra concentrated in degree zero with coproduct $\Delta$ defined by
$$
\Delta(1) = 1\otimes 1.
$$
A \emph{morphism} of graded coalgebras is a linear map $\varphi \colon C \to C'$ of degree zero such that
$$
(\varphi\otimes \varphi) \circ \Delta = \Delta' \circ \varphi\quad \text{and}\quad \varepsilon = \varepsilon'\circ \varphi.
$$
A \emph{coaugmented} graded coalgebra is a graded coalgebra $C$ together with a morphism of graded coalgebras
$$
\eta \colon \R \to C.
$$

A \emph{coderivation of degree $k$} for a graded coalgebra $C$ is a linear map $\partial \colon C \to C$ of degree $k$ such that 
$$
\Delta \circ \partial = (\partial \otimes \mathrm{id} + \mathrm{id}\otimes \partial) \circ \Delta\quad \text{and}\quad \varepsilon\circ \partial = 0. 
$$
A \emph{differential graded coalgebra} is a graded coalgebra $C$ equipped with a coderivation $\partial \colon C \to C$ of degree~$-1$ such that $\partial \circ \partial = 0$. 

\subsection{Example}
Let $X$ be a topological space and let 
$$
C_*(X; \R) = \bigoplus_{n\geqslant 0} C_n(X;\R)
$$
be the graded vector space of singular chains on $X$.  Then $C_*(X; \R)$ has the structure of a differential graded coalgebra with coproduct given by the Alexander--Whitney diagonal
$$
\Delta \colon C_*(X; \R) \to C_*(X; \R) \otimes C_*(X; \R)
$$
and with counit $\varepsilon\colon C_*(X; \R) \to \R$ defined by
$$
\varepsilon(\sigma) =
\begin{cases}
1 & \text{if } \deg(\sigma) = 0,\\
0 & \text{otherwise,}
\end{cases}
$$
for any singular simplex $\sigma$.  Moreover, a choice of a basepoint $x\in X$ yields a coaugmentation
$$
\eta_x \colon \R \to C_*(X; \R)
$$
defined by $\eta_x(1) = [x] \in C_0(X; \R)$.

\subsection{The cobar construction}

Let $C$ be a coaugmented differential graded coalgebra. The \emph{coaugmentation coideal of $C$} is $\bar{C} = \mathrm{Coker}(\eta)$, where $\eta$ is the coaugmentation of $C$.  The \emph{reduced coproduct $\bar{\Delta}$ of $C$} is the composite 
$$
\bar{\Delta} = (q\otimes q )\circ \Delta\colon C \to \bar{C}\otimes \bar{C},
$$
where $q\colon C\to \bar{C}$ is the quotient map. The \emph{cobar complex} of a coaugmented differential graded coalgebra~$C$ is the augmented differential graded algebra
$$
\Omega (C) = (T(s^{-1}\bar{C}), d_\Omega, \epsilon)
$$
with differential $d_\Omega$ defined by
\begin{equation*}
d_\Omega (s^{-1}c) = -s^{-1}\partial c + \sum(-1)^{\mathrm{deg}(c_{(1)})}s^{-1}c_{(1)}\otimes s^{-1}c_{(2)},
\end{equation*}
where $\bar{\Delta}(c) = \sum c_{(1)}\otimes c_{(2)}$, and with $\epsilon$ the canonical augmentation. 

An element $x \in \Omega(C)$ is a \emph{pure tensor of length $k$} if there exist elements~$c_1, \ldots c_k \in C$ such that 
$$
x = s^{-1}c_1\otimes \cdots \otimes s^{-1}c_k.
$$
In that case, we write 
$$
x = [c_1\vert \cdots \vert c_k].
$$
By convention, an element $x\in \Omega(C)$ is \emph{homogeneous of length $0$} if $x\in \R$. 

\subsection{Formal power series connections}\label{Subsection_formalpowerseriesconnections}

Let $(C_{dR}^*(M), d)$ be the dga of differential forms on a smooth manifold $M$ and let $(A, \delta)$ be an augmented dga with augmentation $\varepsilon \colon A\to \R$.  Let~\mbox{$C_{dR}^*(M)\otimes A$} be the tensor product graded algebra, that is,
$$
(C_{dR}^*(M)\otimes A)_n = \bigoplus_{i+j = n}C_{dR}^i(M) \otimes A_j
$$
for all $n\geqslant 0$.  The product of two elements of $C_{dR}^*(M)\otimes A$ is defined by 
$$
\left(\sum_{i\in I} \omega_i \otimes a_i \right) \wedge \left(\sum_{j\in J} \psi_j \otimes b_j\right) = \sum_{i\in I}\sum_{j \in J} (\omega_i \wedge \psi_j) \otimes a_i b_j. 
$$
Note that this is not the usual Koszul-signed tensor product.

Let $F_\bullet$ be the descending filtration of $C_{dR}^*(M)\otimes A$ defined by
$$
F_n  =\bigoplus_{k\geqslant n} (C_{dR}^*(M)\otimes A)_k
$$
for all $n\geqslant 0$.  The \emph{degree completion} of $C_{dR}^*(M)\otimes A$ is the algebra
$$
C_{dR}^*(M)\, \widehat{\otimes}\,  A = \varprojlim_n\frac{C_{dR}^*(M)\otimes A}{F_n}.
$$
We denote elements of $C_{dR}^*(M)\, \widehat{\otimes}\,  A$ as formal sums
$$
\sum_{i\geqslant 0} \omega_i\otimes a_i
$$
of homogeneous elements $\omega_i \in C^*_{dR}(M)$ and $a_i\in A$ such that 
$$
\deg(\omega_i) +\deg(a_i) \leqslant \deg(\omega_{i+1}) + \deg(a_{i+1}) 
$$
for all $i\geqslant 0$.   We extend $d$ and $\delta$ to the degree completion via
\begin{align*}
d\left(\sum_{i\geqslant 0}\omega_i \otimes a_i \right) &= \sum_{i\geqslant 0}d\omega_i \otimes a_i,\\
\delta\left(\sum_{i\geqslant 0}\omega_i \otimes a_i\right) &= \sum_{i\geqslant 0}\omega_i \otimes \delta a_i.
\end{align*}
Let $\varphi : C_{dR}^*(M) \to C_{dR}^*(M)$ be the automorphism defined by $\varphi (\omega) = (-1)^{\mathrm{deg}(\omega)}\omega$ for any homogeneous element $\omega \in C_{dR}^*(M)$. We extend $\varphi$ to $C_{dR}^*(M)\,  \widehat{\otimes}\, A$ via:
\begin{equation*}
\varphi\left(\sum_{i\geqslant 0}\omega_i \otimes a_i\right) = \sum_{i\geqslant 0}\varphi(\omega_i) \otimes a_i.
\end{equation*}

A \emph{formal power series connection} is an element $\displaystyle \omega = \sum_{i\geqslant 0}\omega_i \otimes a_i \in C_{dR}^*(M)\, \widehat{\otimes}\, A$ such that
\begin{itemize}
\item $a_i \in \mathrm{Ker}(\varepsilon)$ for all $i\geqslant 0$,
\item $\deg (\omega_i) = \deg(a_i) +1$ for all $i\geqslant 0$,
\item the \emph{twisted cochain relation} is verified: $d\omega + \delta\omega = \varphi(\omega) \wedge \omega$.
\end{itemize}

\subsection{Example: a formal power series connection on the circle}Consider the circle $S^1$, and let~$\R[X]$ be the algebra of polynomials in $X$ with augmentation $\varepsilon \colon \R[X] \to \R$ defined by $\varepsilon(X) = 0$.  Equip the algebra $\R[X]$ with trivial grading and differential. Define
$$
\omega = \eta\otimes X \in C^*_{dR}(S^1)\, \widehat{\otimes}\, \R[X],
$$
where $\eta$ is any closed 1-form.  Then $\omega$ is a formal power series connection. 

\subsection{The transport of a formal power series connection}\label{transport} 
Let $(A, \delta)$ be an augmented dga and let~$B^*(M)$ be the bar complex of iterated integrals on a smooth manifold $M$.  Let $B^*(M)\, \widehat{\otimes}\,  A$ be the degree completion of $B^*(M)\otimes A$, that is,
$$
B^*(M)\, \widehat{\otimes}\,  A = \varprojlim_n \frac{B^*(M) \otimes A}{\bigoplus_{k\geqslant n} (B^*(M) \otimes A)_k}.
$$
Let $F_\bullet$ be the descending filtration of $B^*(M)\, \widehat{\otimes}\,  A$ defined by
$$
F_n = B^*(M) \, \widehat{\otimes}\, \left(\mathrm{Ker}(\varepsilon)\right)^n.
$$
The \emph{augmentation completion} of $B^*(M)\, \widehat{\otimes}\,  A$ is the algebra $B^*(M)\, \widehat{\otimes}^{\varepsilon}\,  A$ defined by
$$
B^*(M)\, \widehat{\otimes}^{\varepsilon} A = \varprojlim_n \frac{B^*(M)\, \widehat{\otimes}\,  A}{F_n}.
$$

The \emph{transport} $T_\omega$ of a formal power series connection
$$
\omega = \sum_{i\geqslant 0} \omega_i \otimes a_i \in C^*_{dR}(M) \, \widehat{\otimes}\, A
$$
is the element 
$$
T_\omega =  1 + \sum_{k\geqslant 1} \sum_{i_1, \ldots i_k \geqslant 0}\int \omega_{i_1}\ldots \omega_{i_k} \otimes a_{i_1}\cdots a_{i_k} \in B^*(M)\, \widehat{\otimes}^\varepsilon A.
$$
It follows from the twisted cochain relation that 
$$
dT_\omega = \delta T_\omega.
$$
Notice that if $\mathrm{deg}(\omega_i) \geqslant 2$ for all $i\geqslant 0$, then $T_\omega \in B^*(M)\, \widehat{\otimes}\, A$. 

\begin{lem}
Let $M$ be a smooth manifold and let $\omega\in C^*_{dR}(M) \, \widehat{\otimes}\, A$ be a formal power series connection with transport $T_\omega$. Then, for every bounded plot $\phi$ of the pointed path space of $M$, the integration of iterated integrals defines an element
$$
\langle T_\omega, \phi\rangle \in \varprojlim_n A/(\mathrm{Ker}(\varepsilon))^n.
$$
\end{lem}

\begin{proof}
Since $a_i \in \mathrm{Ker}(\varepsilon)$ for all $i\geqslant 0$, then modulo $(\mathrm{Ker}(\varepsilon))^n$ only terms of length less than $n$ contribute.  Let $\phi \colon U \to P(M; x_0, x_1)$ be a bounded plot. Since the domain of $\phi$ is finite-dimensional,  we have
$$
\tau_n= \delta_{\dim (U)}^0+\sum_{k=1}^{n-1}\sum_{i_1,\ldots,i_k\geqslant0} \langle\int\omega_{i_1}\cdots\omega_{i_k},\phi\rangle a_{i_1}\cdots a_{i_k} \in A/(\mathrm{Ker}(\varepsilon))^n.
$$
Moreover,  $\tau_n - \tau_m \in (\mathrm{Ker}(\varepsilon))^n$ for all $n\leqslant m$, therefore the element
$$
\langle T_\omega, \phi\rangle = (\tau_n)_{n\geqslant 1} \in \varprojlim_n A/(\mathrm{Ker}(\varepsilon))^n
$$
is well defined. 
\end{proof}

\subsection{Example: the homology of loop spaces}

Let $M$ be a simply connected smooth manifold and let
$$
H_*(M; \R) = \bigoplus_{n\geqslant 0}H_n(M; \R)
$$
be the graded coalgebra given by the singular homology of $M$ with real coefficients.  Let $\Omega(H_*(M; \R))$ be the cobar complex on $H_*(M; \R)$ (equipped with the canonical coaugmentation). Suppose that~$H_*(M; \R)$ is of finite type and let $\{c_i\}_{i\in I}$ be a basis of $H_{\geqslant 1}(M;\R)$.  Let $\{\omega_i\}_{i\in I}$ be closed forms on $M$ representing cohomology classes that are dual to $\{c_i\}_{i\in I}$. Chen \cite{chen1977} shows that there exists a derivation~$\delta$ on~$\Omega(H_*(M; \R))$ and a formal power series connection
$$
\omega \in C^*_{dR}(M)\otimes \Omega(H_*(M; \R))
$$
whose tensor-length-one component is
$$
\omega^{(1)} = \sum_{i\in I} \omega_i \otimes [c_i].
$$
Let $\Omega M$ be the based loop space of $M$. Then, the map
$$ 
\begin{array}{ccl}
C_*(\Omega M; \R) & \to & \Omega (H_*(M; \R))\\
 c & \mapsto & \langle T_\omega, c\rangle,
\end{array}
$$
given by the integration of the transport of $\omega$ over $c$ induces an isomorphism on homology:
$$
H_{*}(\Omega M; \R) \, \tilde{\longrightarrow}\, H_{*}(\Omega(H_*(M; \R)), \delta).
$$
We refer the reader to \cite{chen1977} for more details.

\section{Higher holonomy functors}\label{higherholonomyfunctors}

\noindent In this section, we show how the transport of a formal power series connection induces functors from the path groupoid, the path 2-groupoid, and the path 3-groupoid of a manifold.  

Throughout this section,  let $A$ be an augmented dga with augmentation $\varepsilon \colon A \to \R$, and let
$$
\omega \in C_{dR}^*(M) \, \widehat{\otimes }\, A
$$
be a formal power series connection on a manifold~$M$ (see Section \ref{Subsection_formalpowerseriesconnections}). We denote by~$T_\omega$ the transport of $\omega$ (see Section \ref{transport}). 

\subsection{Holonomy functors}

Let $\omega_1, \ldots , \omega_k$ be differential forms on $M$.  Kohno shows \cite{kohno2016higher} that if $\gamma$ and $\gamma'$ are rank-1 homotopic 1-paths, then:
\begin{equation*}
\langle \int \omega_1 \ldots  \omega_k, \gamma\rangle  = \langle \int  \omega_1 \ldots \omega_k, \gamma' \rangle,
\end{equation*}
where $\langle \cdot, \cdot\rangle$ denotes the integration of iterated integrals over bounded plots (see Section \ref{iteratedintegrals}).
Moreover, the composition of 1-tracks is compatible with integration: if $\gamma$ and $\delta$ are 1-tracks such that $\gamma(1) = \delta(0)$, then
$$
\langle T_\omega, \gamma\circ \delta \rangle = \langle T_\omega, \gamma\rangle \langle T_\omega, \delta\rangle.
$$

Let $\widehat{A}^\varepsilon$ be the augmentation completion of $A$, that is,
$$
\widehat{A}^\varepsilon = \varprojlim_n A/(\mathrm{Ker}(\varepsilon))^n
$$
and let $\widehat{A}^\varepsilon_0$ be the degree-$0$ subalgebra, see Section \ref{augmentation completion}. We define the category $\mathcal{C}_1(A)$ by the following data:
\begin{itemize}
\item the category $\mathcal{C}_1(A)$ has a single object $*$,
\item an endomorphism of $*$ in $\mathcal{C}_1(A)$ is an element of $\widehat{A}^\varepsilon_0$ with composition given by the product in $\widehat{A}^\varepsilon_0$.
\end{itemize}

The \emph{holonomy functor of $\omega$} is the functor
$$
\mathrm{Hol}^1_\omega \colon \mathcal{P}_1(M) \to \mathcal{C}_1(A)
$$
defined by
$$ \mathrm{Hol}^1_\omega ([\gamma]) = \langle T_\omega, \gamma\rangle,$$
for any 1-track $[\gamma]$.

\subsection{Integration over 2-tracks}

In this section, we extend the notion of integration over 2-paths to that of integration over 2-tracks.  Moreover, we show that the integration of the transport $T_\omega$ of a formal power series connection $\omega$ over a 2-track is compatible with the vertical and the horizontal composition. Let $\widehat{A}^\varepsilon_1$ be the subspace of $\widehat{A}^\varepsilon$ spanned by homogeneous elements of degree 1 (see Section \ref{augmentation completion})

\begin{lem}\label{lem2tracks}
If $g$ and $h$ are rank-2 homotopic 2-paths in $M$, then
$$ \langle T_\omega, g\rangle - \langle T_\omega, h\rangle \in \delta\left( \widehat{A}^\varepsilon_1 \widehat{A}^\varepsilon_1\right). $$
\end{lem}
\begin{proof}
Let $g-h = \partial H$ with $H$ a rank-2 homotopy.
\begin{align*}
\langle T_\omega, g\rangle - \langle T_\omega, h\rangle 
= \langle T_\omega, \partial H\rangle 
= \langle  dT_\omega, H\rangle
= \delta \langle T_\omega, H\rangle,
\end{align*}
where the second equality is a consequence of Stokes' theorem, and the third equality is a consequence of the twisted cochain condition. Since $H$ is rank-2, $\langle T_\omega, H\rangle \in \widehat{A}^\varepsilon_1 \widehat{A}^\varepsilon_1$.
\end{proof}

\begin{rem*}Contrary to \cite[Proposition 3.4]{Koh2022}, we find that if $g$ and $h$ are rank-2 homotopic 2-paths and~$\omega_1, \ldots, \omega_k$ are differential forms of degree $\deg(\omega_1), \ldots, \deg(\omega_k)$, respectively, then we do not necessarily have
$$
\langle \int \omega_1\ldots\omega_k , g\rangle = \langle \int \omega_1\ldots\omega_k , h\rangle.
$$
Indeed, we need a stronger condition: that $g$ and $h$ are laminated rank-2 homotopic (see Lemma \ref{lemlam2tracks}). 
\end{rem*}

The vertical composition of 2-paths is compatible with integration:
\begin{lem}\label{lemvert2tracks}
If $g$ and $h$ are 2-paths such that $g(1,-) = h(0,-)$, then
$$ 
\langle T_\omega,  g \cdot h\rangle = \langle T_\omega, g\rangle + \langle T_\omega, h\rangle.
$$
\end{lem}
\begin{proof}
For any differential forms $\omega_1, \ldots \omega_k$ of positive degree such that $\deg(\omega_1) + \cdots +\deg(\omega_k) -k = 1$, we have: 
\begin{align*}
\langle \int \omega_1 \ldots \omega_k ,  g\cdot h\rangle 
&= \int_{[0,\frac{1}{2}]}\int_{\Delta^k}\widetilde{g\cdot h}^*(\omega_1\times \cdots \times \omega_k) + \int_{[\frac{1}{2},1]}\int_{\Delta^k}\widetilde{g\cdot h}^*(\omega_1\times \cdots \times \omega_k),\\
&= \int_{[0,1]}\int_{\Delta^k}\widetilde{g}^*(\omega_1\times \cdots \times \omega_k) + \int_{[0,1]}\int_{\Delta^k}\widetilde{h}^*(\omega_1\times \cdots \times \omega_k),\\
&= \langle \int \omega_1 \ldots \omega_k ,  g\rangle + \langle \int \omega_1 \ldots \omega_k ,  h\rangle. \qedhere
\end{align*}
\end{proof}
We have the following formulas for the horizontal composition of 2-paths.
\begin{lem}\label{lemhoriz2tracks}
If $g$ and $h$ are 2-paths such that $g(-, 1) = h(-,0)$, then the following equations hold modulo~\mbox{$\delta\left( \widehat{A}^\varepsilon_1 \widehat{A}^\varepsilon_1\right)$}:
\begin{align*}
\langle T_\omega, g\circ h\rangle &= \langle T_\omega, g\rangle \langle T_\omega,  h(0,-) \rangle + \langle T_\omega, g(1,-) \rangle \langle T_\omega, h\rangle,\\
&= \langle T_\omega, g(0,-)\rangle \langle T_\omega, h\rangle + \langle T_\omega, g\rangle \langle T_\omega, h(1,-)\rangle. 
\end{align*}
\end{lem}
\begin{proof}
Note that $g\circ h$ is rank-2 homotopic to $(g\circ c_{h(0,-)}) \cdot (c_{g(1,-)}\circ h)$ and to $(c_{g(0,-)}\circ h) \cdot (g \circ c_{h(1,-)})$.  Therefore, the equations result from Lemma \ref{lem2tracks}, Lemma \ref{lemvert2tracks} and  \cite[Proposition 2.1.1]{chen1973}.
\end{proof}

\subsection{The target 2-category}

Recall that $\widehat{A}^\varepsilon$ is the augmentation completion of $A$, and that~$\widehat{A}^\varepsilon_k$ is the vector subspace spanned by homogeneous elements of degree $k$. We define the strict 2-category $\mathcal{C}_2(A)$ by the following data:
\begin{itemize}
\item the 2-category $\mathcal{C}_2(A)$ has a single object $*$,
\item an endomorphism of $*$ is an element of $\widehat{A}^\varepsilon_0$, with composition given by the product in~$\widehat{A}^\varepsilon_0$,
\item for any 1-morphisms $m$ and $n$ of $\mathcal{C}_2(A)$,  a 2-morphism from $m$ to $n$ is an element of the set
$$\mathcal{C}_2(A)(m,n) = \left\lbrace M \in \widehat{A}^\varepsilon_1/\delta\left( \widehat{A}^\varepsilon_1\widehat{A}^\varepsilon_1 \right) \mid \delta M = m - n \right\rbrace,$$
\item the vertical composition of 2-morphisms is given by
\[
\begin{array}{ccl}
\mathcal{C}_2(A)(m,n)\times \mathcal{C}_2(A)(n,p) & \to & \mathcal{C}_2(A)(m,p)\\
(M, N) & \mapsto & M+ N,
\end{array}
\]
\item the horizontal composition of 2-morphisms is given by
\[
\begin{array}{ccl}
\mathcal{C}_2(A)(m,n)\times \mathcal{C}_2(A)(p,q) & \to & \mathcal{C}_2(A)(m p,n q)\\
(M,N) & \mapsto & Mp + nN = mN + Mq.
\end{array}
\]
\end{itemize}
Note that the identity 
$$
Mp + nN = mN + Mq
$$
only holds modulo $\delta\left( \widehat{A}^\varepsilon_1\widehat{A}^\varepsilon_1\right)$.  It is a straightforward verification that $\mathcal{C}_2(A)$ is a strict 2-category. 

\subsection{Holonomy 2-functors}

Recall from Section \ref{formalpowerseriesconnections} the transport $T_\omega$ of the formal power series connection $\omega$, and from Section \ref{iteratedintegrals} the integration $\langle \cdot, \cdot\rangle$ of iterated integrals over bounded plots.  

Let's define the association $\mathrm{Hol}^2_\omega$. For any 1-track $[\gamma]$,  we set
$$ \mathrm{Hol}^2_\omega ([\gamma]) = \langle T_\omega , \gamma\rangle.$$
For any 2-track $[g]$, we set
$$ \mathrm{Hol}^2_\omega ([g]) = [\langle T_\omega, g\rangle ],$$
where $[\langle T_\omega, g\rangle ]$ denotes the class of $\langle T_\omega, g\rangle$ in $ \widehat{A}^\varepsilon_1/ \delta\left( \widehat{A}^\varepsilon_1\widehat{A}^\varepsilon_1\right)$.
We deduce from Lemma \ref{lem2tracks} that the above association is well defined. 

From Lemmas \ref{lemvert2tracks} and \ref{lemhoriz2tracks}, we deduce  that the association~$\mathrm{Hol}^2_\omega$ is compatible with the vertical and horizontal composition of 2-tracks:
\begin{thm}\label{thm2holonomy}
The transport of the formal power series connection $\omega$ induces a 2-functor:
$$ \mathrm{Hol}^2_\omega \colon \mathcal{P}_2(M) \to \mathcal{C}_2(A)$$ 
called the \emph{holonomy 2-functor of $\omega$}.
\end{thm}

We show in Section \ref{subsection:non-triviality} that the holonomy 2-functor can detect non-trivial homotopy classes.

\subsection{Integration over laminated 2-tracks}
We now focus on the construction of holonomy \mbox{3-functors}. In this section, we extend the notion of integration over 2-paths to that of integration over laminated \mbox{2-tracks}. Moreover, we show that the integration of the transport $T_\omega$ of the formal power series connection~$\omega$ over laminated \mbox{2-tracks} is compatible with the vertical composition and the two horizontal compositions of laminated 2-tracks. 

\begin{lem}\label{lemlam2tracks}
If $g$ and $h$ are laminated rank-2 homotopic 2-paths, then 
\begin{equation*}
\langle T_\omega, g\rangle = \langle T_\omega, h\rangle. 
\end{equation*}
\end{lem}
\begin{proof}
Let $H$ be a laminated rank-2 homotopy from $g$ to $h$.  We have
\begin{align*}
\langle T_\omega, h\rangle - \langle T_\omega, g\rangle 
&= \langle T_\omega, \partial H\rangle, \\
&= \langle dT_\omega, H\rangle,
\end{align*}
where the second equality is consequence of Stokes' theorem.  Therefore, it suffices to show that for any positive integer $k$ and for any differential forms $\omega_1, \ldots, \omega_k$ of positive degree on $M$ such that $$\deg(\omega_1) + \cdots + \deg(\omega_k) = k + 2,$$ we have $\displaystyle \langle \int \omega_1\ldots \omega_k , H\rangle = 0$.  Recall that 
\begin{align*}
\langle \int \omega_1\ldots\omega_k , H\rangle 
&= \int_{[0,1]^2}\int_{\Delta^k}\widetilde{H}^*(\omega_1\times \cdots \times \omega_k),
\end{align*}
where $\widetilde{H}(r,s,t_1, \ldots, t_k) = (H(r,s,t_1), \ldots, H(r,s,t_k))$ for all $(r,s)\in [0,1]^2$ and $(t_1, \ldots, t_k) \in \Delta^k$. The condition that $H$ is laminated ensures that the rank of $d\widetilde{H}$ is at most $k+1$, and therefore 
\begin{equation*}
\widetilde{H}^*(\omega_1\times \cdots\times \omega_k) = 0.\qedhere
\end{equation*}

\end{proof}

The left and right whiskering of 2-paths with 1-paths is compatible with integration:
\begin{lem}\label{lemw2tracks}
If $\gamma$ is a 1-path and $g$ is a 2-path such that $\gamma(1) = g([0,1]\times \{0\})$, then
\begin{equation*}
\langle T_\omega, \gamma\circ g \rangle = \langle T_\omega, \gamma\rangle \langle T_\omega, g\rangle.
\end{equation*}
If $\gamma$ is a 1-path and $g$ is a 2-path such that $g([0,1]\times \{1\}) = \gamma(0)$, then
\begin{equation*}
\langle T_\omega, g \circ \gamma\rangle = \langle T_\omega, g\rangle\langle T_\omega, \gamma\rangle.
\end{equation*}
\end{lem}
\begin{proof}
This is a direct consequence of \cite[Proposition 2.1.2]{chen1973}.
\end{proof}

As a consequence of Lemmas \ref{lemvert2tracks} and \ref{lemw2tracks}, we get:
\begin{lem}\label{lemhoriz2paths}
If $h_1$ and $h_2$ are 2-paths such that $h_1([0,1]\times \{1\}) = h_2([0,1]\times \{0\})$, then
\begin{align*}
\left\langle T_\omega, \begin{pmatrix}
& h_2\\
h_1 & 
\end{pmatrix}\right\rangle &= \langle T_\omega, h_1\rangle \langle T_\omega, h_2(0,-)\rangle + \langle T_\omega, h_1(1,-)\rangle \langle T_\omega, h_2\rangle, \\
\left\langle T_\omega, \begin{pmatrix}
h_1& \\
& h_2
\end{pmatrix}\right\rangle & = \langle T_\omega, h_1(0,-)\rangle \langle T_\omega, h_2\rangle + \langle T_\omega, h_1\rangle \langle T_\omega, h_2(1,-)\rangle.
\end{align*}
\end{lem}

\subsection{Integration over 3-tracks}
In this section, we extend the notion of integration over good 3-paths to that of integration over 3-tracks. Moreover, we show that the integration of the transport $T_\omega$ of the formal power series connection $\omega$ is compatible with the upward composition, the vertical composition and the two horizontal compositions of 3-tracks. 

We denote by $\widehat{A}^\varepsilon_{1,2}$ the subspace of $\widehat{A}^\varepsilon_3$ defined by
$$
\widehat{A}^\varepsilon_{1,2}=  \left(\widehat{A}^\varepsilon_1\widehat{A}^\varepsilon_2\right) +  \left(\widehat{A}^\varepsilon_2\widehat{A}^\varepsilon_1 \right)
$$ 
\begin{lem}\label{lem3tracks}
If $J$ and $J'$ are rank-3 homotopic good 3-paths, then
\begin{equation*}
\langle T_\omega, J'\rangle - \langle T_\omega, J\rangle \in \delta \widehat{A}^\varepsilon_{1,2}.
\end{equation*}
\end{lem}
\begin{proof}
Let $W$ be a rank-3 homotopy from $J$ to $J'$.  We have
\begin{align*}
\langle T_\omega, J'\rangle - \langle T_\omega, J\rangle = \langle T_\omega, \partial W\rangle = \langle dT_\omega, W\rangle = \delta\langle T_\omega, W\rangle,
\end{align*}
where the second equality is a consequence of Stokes' theorem, and the third equality is a consequence of the twisted cochain condition. Since $W$ is rank-3, $\langle T_\omega, W\rangle \in \widehat{A}^\varepsilon_{1,2}$.
\end{proof}

\begin{lem}\label{lemup3paths}
If $J$ and $J'$ are 3-paths such that $J(1, -,-) = J'(0,-,-)$, then 
\begin{equation*}
\langle T_\omega,  J* J'\rangle = \langle T_\omega, J\rangle + \langle T_\omega, J'\rangle.
\end{equation*}
\end{lem}
\begin{proof}
For any differential forms $\omega_1, \ldots, \omega_k$ of positive degree such that $\deg(\omega_1)+\cdots +\deg(\omega_k) - k = 2$,  we have:

\begin{align*}
\langle\int \omega_1\ldots\omega_k, J*J'\rangle 
&= \int_{[0, \frac{1}{2}]\times [0,1]}\int_{\Delta^k}\widetilde{J*J'}^*(\omega_1\times\cdots\times \omega_k) + \int_{[\frac{1}{2}, 1]\times [0,1]}\int_{\Delta^k}\widetilde{J*J'}^*(\omega_1\times\cdots\times \omega_k),\\
&= \int_{[0, 1]\times [0,1]}\int_{\Delta^k}\widetilde{J}^*(\omega_1\times\cdots\times \omega_k) +  \int_{[0, 1]\times [0,1]}\int_{\Delta^k}\widetilde{J'}^*(\omega_1\times\cdots\times \omega_k),\\
&= \langle\int \omega_1\ldots\omega_k, J\rangle  + \langle\int \omega_1\ldots\omega_k, J'\rangle .\qedhere
\end{align*}
\end{proof}
By exchanging the first two coordinates in the previous proof, we get:
\begin{lem}\label{lemvert3paths}
If $J$ and $J'$ are 3-paths such that $J(-,1,-) = J'(-,0,-)$, then
$$\langle T_\omega,  J\cdot J'\rangle = \langle T_\omega, J\rangle + \langle T_\omega, J'\rangle. $$
\end{lem}

As a consequence of Lemma \ref{lemvert3paths} and \cite[Proposition 2.1.2]{chen1973}, we get:
\begin{lem}\label{lemhoriz3paths}
If $J$ and $J'$ are good 3-paths such that $J([0,1]^2\times \{1\}) = J'([0,1]^2\times \{0\})$, then
\begin{align*}
\left\langle T_\omega, \begin{pmatrix}
& J'\\
J & 
\end{pmatrix}\right\rangle &= \langle T_\omega, J\rangle \langle T_\omega, J'(0,0,-)\rangle + \langle T_\omega J(0,1,-)\rangle \langle T_\omega, J'\rangle,\\
\left\langle T_\omega, \begin{pmatrix}
J& \\
& J'
\end{pmatrix}\right\rangle &= \langle T_\omega, J(0,1,-)\rangle \langle T_\omega, J'\rangle + \langle T_\omega J\rangle \langle T_\omega, J'(0,0,-)\rangle.\\
\end{align*}
\end{lem}

\subsection{The target Gray 3-category}

Recall that $\widehat{A}^\varepsilon$ is the augmentation completion of $A$, and that~$\widehat{A}^\varepsilon_k$ is the vector subspace generated by homogeneous elements of degree $k$ (see Section~\ref{formalpowerseriesconnections}). We define the Gray 3-category $\mathcal{C}_3(A)$ by the following data:
\begin{itemize}
\item the Gray 3-category $\mathcal{C}_3(C)$ has a single object $*$,
\item an endomorphism of $*$ is an element of $\widehat{A}^\varepsilon_0$ with composition given by the product,
\item for 1-morphisms $m$ and $n$ of $\mathcal{C}_3(A)$,  a 2-morphism from $m$ to $n$ is an element of the set
$$
\mathcal{C}_3(A)(m,n) = \left\lbrace M\in \widehat{A}^\varepsilon_1\mid \delta M = m-n \right\rbrace, 
$$
\item the vertical composition of 2-morphisms is given by
\[
\begin{array}{ccl}
\mathcal{C}_3(A)(m,n) \times \mathcal{C}_3(A)(n,p) & \to & \mathcal{C}_3(A)(m,p)\\
(M,N) & \mapsto & M + N,
\end{array}
\]
\item the left whiskering of a 2-morphism $M$ with a 1-morphism $m$ is given by $mM$,
\item the right whiskering of a 2-morphism $M$ with a 1-morphism $m$ is given by $Mm$,
\item for 2-morphisms $M,N \in \mathcal{C}_3(A)(m,n)$, a 3-morphism from $M$ to $N$ is an element of the set
$$
\mathcal{C}_3(A)(m,n)(M,N) = \left\lbrace \alpha \in  \widehat{A}^\varepsilon_2/\delta \widehat{A}^\varepsilon_{1,2} \mid \delta \alpha = M-N \right\rbrace,
$$
\item the upward composition of 3-morphisms is given by
\[
\begin{array}{ccl}
\mathcal{C}_3(A)(m,n)(M,N) \times \mathcal{C}_3(A)(m,n)(N,P) & \to & \mathcal{C}_3(A)(m,n)(M,P)\\
(\alpha, \beta) & \mapsto & \alpha + \beta,
\end{array}
\]
\item the vertical composition of 3-morphisms is given by
\[
\begin{array}{ccl}
\mathcal{C}_3(A)(m,n)(M,N) \times \mathcal{C}_3(A)(n,p)(P,Q) & \to & \mathcal{C}_3(A)(m,p)(M\cdot P,N\cdot Q)\\
(\alpha, \beta) & \mapsto & \alpha + \beta,
\end{array}
\]
\item the left whiskering of a 3-morphism $\alpha$ with a 1-morphism $m$ is given by $m\alpha$,
\item the right whiskering of a 3-morphism $\alpha$ with a 1-morphism $m$ is given by $\alpha m$.
\end{itemize}
It is a straightforward verification that $\mathcal{C}_3(A)$ is a Gray 3-category.

\subsection{Holonomy 3-functors}

Recall from Section \ref{formalpowerseriesconnections} the transport $T_\omega$ of the formal power series connection $\omega$, and from Section \ref{iteratedintegrals} the integration $\langle \cdot, \cdot \rangle$ of iterated integrals over bounded plots.

Let's define the association $\mathrm{Hol}^3_\omega$. For any 1-track $[\gamma]$, we set
$$\mathrm{Hol}^3_\omega([\gamma]) = \langle T_\omega, \gamma\rangle.$$
For any laminated 2-track $[g]$, we set 
$$\mathrm{Hol}^3_\omega([g]) = \langle T_\omega, g\rangle.$$
We deduce from Lemma \ref{lemlam2tracks} that the element $\langle T_\omega, g\rangle \in \widehat{A}^\varepsilon_1$ does not depend on the choice of representative 2-path $g$.
For any 3-track $[J]$, we set 
$$\mathrm{Hol}^3_\omega([J]) = [\langle T_\omega, J\rangle],$$
where $[\langle T_\omega, J\rangle]$ denotes the class of $\langle T_\omega, J\rangle$ in $\widehat{A}^\varepsilon_2/\delta\widehat{A}^\varepsilon_{1,2}$. We deduce from Lemma \ref{lem3tracks} that the class~$[\langle T_\omega, J\rangle]$ does not depend on the choice of representative good 3-path $J$. 

From Lemmas \ref{lemvert2tracks} and \ref{lemhoriz2paths},  we deduce that the association $\mathrm{Hol}^3_\omega$ is compatible with the vertical composition and the horizontal compositions of laminated 2-tracks. From Lemmas \ref{lemup3paths}, \ref{lemvert3paths} and  \ref{lemhoriz3paths}, we deduce 
that the association $\mathrm{Hol}^3_\omega$ is compatible with the upward composition, the vertical composition, and the horizontal compositions of 3-tracks:

\begin{thm}\label{thm3holonomy}
The transport of the formal power series connection $\omega$ induces a Gray functor:
$$
 \mathrm{Hol}^3_\omega \colon \mathcal{P}_3(M) \to \mathcal{C}_3(A)
$$
called the \emph{holonomy 3-functor of $\omega$}.
\end{thm}

We show in Section \ref{subsection:non-triviality} that the holonomy 3-functor can detect non-trivial homotopy classes.  In Section~\ref{application} we study the holonomy 3-functor induced by a formal power series connection on higher configuration spaces. 

\subsection{Non-triviality of the higher holonomy functors}\label{subsection:non-triviality}

In this section we illustrate the non-triviality of the higher holonomy functors on the sphere $S^2$.  Fix an orientation of $S^2$ and let
$$
\eta \in C^2_{dR}(S^2; \R)
$$
be a differential 2-form such that 
$$
\int_{S^2}\eta = 1.
$$
Consider the cobar algebra
$$
A = \Omega (H_*(S^2; \R)) \cong \R[X]
$$
where $X = s^{-1}[S^2]$ is the desuspension of the fundamental class of $S^2$. We equip $A$ with the trivial differential and the augmentation $\varepsilon$ defined by
$$
\varepsilon (X) = 0.
$$
Then, the element
$$
\omega = \eta\otimes X \in C^2_{dR}(S^2; \R)\, \widehat{\otimes}\, A
$$
is a formal power series connection in the sense of Section \ref{Subsection_formalpowerseriesconnections}.  

The 2-category $\mathcal{C}_2(A)$ consists of a single object and its 1-morphisms are given by the set $\R$.  The 2-morphisms are given by 
$$
\mathcal{C}_2(A)(m,n) = \begin{cases}
\R X & \text{if}\ m = n,\\
\emptyset & \text{if}\ m\neq n.
\end{cases}
$$
Let $\sigma \colon [0,1]^2\to S^2$ be a 2-path with constant boundary representing the positively oriented generator of~$\pi_2(S^2)$ and such that
$$
\int_{[0,1]^2}\sigma^*(\eta) = 1.
$$
Then
$$
\mathrm{Hol}^2_\omega([\sigma]) = X,
$$
whereas the holonomy of the constant 2-track is $0$.  Thus $\mathrm{Hol}^2_\omega$ detects 2-tracks with non-trivial class in~\mbox{$\pi_2(S^2) \otimes \R$}. 

Next we consider the holonomy 3-functor. Let $\sigma\times \sigma$ be the 3-path defined by
$$
(\sigma \times \sigma) (r,s,t) = \begin{cases}
\sigma(r,2t) & \text{if}\ 0\leqslant t < \frac{1}{2},\\
\sigma(s, 2t-1) &\text{if} \ \frac{1}{2}\leqslant t \leqslant 1,
\end{cases}
$$
and let $\sigma^{-1}\times \sigma^{-1}$ be the 3-path defined by
$$
(\sigma^{-1} \times \sigma^{-1}) (r,s,t) = \begin{cases}
\sigma(r,1-2t) & \text{if}\ 0\leqslant t < \frac{1}{2},\\
\sigma(s, 2-2t) &\text{if} \ \frac{1}{2}\leqslant t \leqslant 1.
\end{cases}
$$
We have
$$
\langle T_\omega,  \sigma\times \sigma\rangle = \langle T_\omega,  \sigma^{-1}\times \sigma^{-1}\rangle = X^2,
$$
where $T_\omega$ is the transport of the formal power series connection $\omega$.  Let $J$ be the 3-path obtained by horizontal composition:
$$
J = \begin{pmatrix}
\sigma\times \sigma& \\
& \sigma^{-1}\times \sigma^{-1}
\end{pmatrix}.
$$
Then, we have
$$
\langle T_\omega, J\rangle = 2X^2.
$$

The Gray 3-category $\mathcal{C}_3(A)$ has the same 1- and 2-morphisms as $\mathcal{C}_2(A)$, and the 3-endomorphisms of a \mbox{2-morphism} are given by the set $\R X^2$. Note that the restrictions of $J$ to the faces~$s= 0$ and~$s=1$ are families of backtracking loops. Straightening these faces by rank-1 homotopies produces a good 3-path~$J'$ such that 
$$
\mathrm{Hol}^3_\omega ([J']) = 2X^2 \neq 0.
$$
In particular, $\mathrm{Hol}^3_\omega$ detects 3-tracks with non-trivial class in $\pi_3(S^2)\otimes \R$, despite $H_3(S^2; \R) = 0$.

\section{Application: a higher holonomy functor for configuration spaces}\label{application}

\noindent In this section, we apply the construction of Section \ref{higherholonomyfunctors} to a formal power series connection defined by Komendarczyk, Koytcheff and Voli\'c \cite{KKV} on the configuration space of $m$ points in $\R^n$ for $n\geqslant 4$.  We describe the resulting holonomy 3-functor.

\subsection{The formal power series connection}

Let $\bar{\mathcal{D}}(m)$ be Kontsevich's differential graded algebra of oriented diagrams for some integers $n\geqslant 4$ and $m\geqslant 1$ as defined in \cite{lambrechtsvolic, KKV}.  Let $I$ be the formality integration map, and let $\mathcal{B}(m)$ be a basis of $\bar{\mathcal{D}}(m)$ (as a vector space) given by choosing a diagram in each orientation class.  Since $n\geqslant 4$, each graded component of $\bar{\mathcal{D}}(m)$ is finite dimensional, therefore the graded dual $\bar{\mathcal{D}}(m)^*$ is a differential graded coalgebra.  We recall the formal power series connection introduced in \cite{KKV}:
\begin{equation*}
\omega = \sum_{\Gamma \in \mathcal{B}(m)}I(\Gamma)\otimes \frac{[\Gamma^*]}{\lvert \mathrm{Aut}(\Gamma)\rvert} \in C_{dR}^*(\mathrm{Conf}(m,\R^n)) \otimes \Omega(\bar{\mathcal{D}}(m)^*).
\end{equation*}
\begin{rem*} Requiring that $n\geqslant 4$ ensures that~$\omega$ only contains finitely many terms in each degree and no terms of degree 1.  Moreover, since~$\mathrm{Conf}(m,\R^n)$ is finite dimensional,  the above sum is finite. 
\end{rem*}

The transport of $\omega$ is given by:
$$
T_\omega = \sum\limits_{k\geqslant 0} \sum\limits_{\Gamma_1, \,\ldots \Gamma_k \in \mathcal{B}(m)}\int I(\Gamma_1)\ldots I(\Gamma_k) \otimes \frac{[\Gamma_1^*\vert \cdots \vert \Gamma_k^*]}{\lvert \mathrm{Aut}(\Gamma_1)\rvert \cdots \lvert \mathrm{Aut}(\Gamma_k)\rvert} \in B^*(\mathrm{Conf}(m,\R^n))\, \widehat{\otimes}\,  \Omega(\bar{\mathcal{D}}(m)^*).
$$

\subsection{The target Gray 3-category}
Since $\bar{\mathcal{D}}(m)^*$ is connected, the Gray 3-category $\mathcal{C}_3(\Omega\bar{\mathcal{D}}(m)^*)$ admits a simpler description:
\begin{itemize}
\item the 1-morphisms of $\mathcal{C}_3(\Omega\bar{\mathcal{D}}(m)^*)$ are given by the set $\R$, with composition given by multiplication,
\item the 2-morphisms of $\mathcal{C}_3(\Omega\bar{\mathcal{D}}(m)^*)$ are given by linear combinations of diagrams of degree 2, with the vertical composition and the horizontal compositions both given by addition,
\item for any 2-morphisms $X$ and $Y$ from $x$ to $y$, the set of 3-morphisms from $X$ to $Y$ is given by
\begin{equation*}
\mathcal{C}_3(\Omega\bar{\mathcal{D}}(m)^*)(X,Y) = \{\alpha \in \Omega^2(\bar{\mathcal{D}}(m)^*)/d_\Omega \Omega^{1,2}(\bar{\mathcal{D}}(m)^*) \mid d_\Omega(\alpha) = Y - X\},
\end{equation*}
with the upward composition, the vertical composition and the horizontal compositions of \mbox{3-morphisms} all given by addition. 
\end{itemize}

\subsection{The holonomy 3-functor}

By Theorem \ref{thm3holonomy},  the transport of $\omega$ induces a holonomy 3-functor:
\begin{equation*}
\mathrm{Hol}^3_\omega\colon \mathcal{P}_3(\mathrm{Conf}(m,\R^n)) \to \mathcal{C}_3(\Omega\bar{\mathcal{D}}(m)^*).
\end{equation*}
Moreover, for all $n\geqslant 4$, the algebra $\bar{\mathcal{D}}(m)$ is $(n-2)$-connected, therefore
\begin{itemize} 
\item For every $1$-path $\gamma$, 
$$
\mathrm{Hol}^3_\omega([\gamma])=1. 
$$
\item For every laminated $2$-track $[g]$, 
$$
\mathrm{Hol}^3_\omega([g])=0. 
$$
\item If $n\geq 5$, then for every $3$-track $[J]$, \[ \mathrm{Hol}^3_\omega([J])=0. \]
\end{itemize}

We now specialize to $n = 4$. Let $\Gamma_{ij}$ denote the diagram consisting of a single edge joining the external vertices $i$ and $j$.  For degree reasons, we have
$$
\bar{\mathcal D}(m)^3 = \bigoplus_{1\leqslant i<j\leqslant m} \mathbb R\Gamma_{ij}.
$$
Moreover, 
$$
\mathrm{Aut}(\Gamma_{ij}) = \{\mathrm{Id}\}
$$
for all $1\leqslant i<j\leqslant m$. Therefore, the holonomy 3-functor admits a simpler description:
\begin{thm}\label{thm:holonomy-3functor}
Let $J\colon [0,1]^3 \to \mathrm{Conf}(m, \R^4)$ be a good 3-path whose boundary faces are constant. Then 
$$
\mathrm{Hol}^3_\omega([J]) = \sum_{1\leq i<j\leq m} \left(\int_{[0,1]^3}J^*(I(\Gamma_{ij}))\right)[\Gamma_{ij}^*] \in \Omega^2(\bar{\mathcal D}(m)^*)/ d_\Omega\Omega^{1,2}(\bar{\mathcal D}(m)^*).
$$
\end{thm}

Notice that any good 3-path $J$ with constant boundary faces determines a homotopy class
$$
[J]_{\mathrm{top}} \in \pi_3(\mathrm{Conf}(m, \R^4)).
$$
We deduce the following:

\begin{cor}\label{cor:homotopy-isomorphism}
The restriction of the holonomy 3-functor to 3-tracks with constant boundary faces induces an isomorphism
$$
\pi_3\bigl(\mathrm{Conf}(m,\mathbb R^4)\bigr)\otimes\mathbb R \overset{\cong}{\longrightarrow} \R\{ [\Gamma_{ij}^*]\mid 1\leqslant i<j\leqslant m\}.
$$
\end{cor}

\newpage

\bibliographystyle{alpha}
\bibliography{Bibliography.bib}

\begin{thebibliography}{KKV24}

\bibitem[BS04]{baez2004higher}
John Baez and Urs Schreiber.
\newblock Higher gauge theory: 2-connections on 2-bundles.
\newblock {\em arXiv preprint hep-th/0412325}, 2004.

\bibitem[CG02]{cohen2002loop}
Frederick~R. Cohen and Samuel Gitler.
\newblock On loop spaces of configuration spaces.
\newblock {\em Transactions of the American Mathematical Society},
  354(5):1705--1748, 2002.

\bibitem[Che73]{chen1973}
Kuo-Tsai Chen.
\newblock Iterated integrals of differential forms and loop space homology.
\newblock {\em Annals of Mathematics}, 97(2):217--246, 1973.

\bibitem[Che77]{chen1977}
Kuo-Tsai Chen.
\newblock Iterated path integrals.
\newblock {\em Bulletin of the American Mathematical Society}, 83(5):831--879,
  1977.

\bibitem[CS98]{carter_knotted_1998}
J~Scott Carter and Masahico Saito.
\newblock {\em Knotted surfaces and their diagrams}, volume~55.
\newblock 1998.

\bibitem[Dam17]{damiani_journey_2017}
Celeste Damiani.
\newblock A journey through loop braid groups.
\newblock {\em Expositiones Mathematicae}, 35(3):252--285, September 2017.

\bibitem[FHT]{felixRationalHomotopyTheory2001}
Yves Félix, Stephen Halperin, and Jean-Claude Thomas.
\newblock {\em Rational {{Homotopy Theory}}}, volume 205 of {\em Graduate
  {{Texts}} in {{Mathematics}}}.
\newblock Springer New York.

\bibitem[FMP10]{martins2010two}
Jo{\~a}o Faria~Martins and Roger Picken.
\newblock On two-dimensional holonomy.
\newblock {\em Transactions of the American Mathematical Society},
  362(11):5657--5695, 2010.

\bibitem[FMP11]{faria_martins_fundamental_2011}
João Faria~Martins and Roger Picken.
\newblock The fundamental {Gray} 3-groupoid of a smooth manifold and local
  3-dimensional holonomy based on a 2-crossed module.
\newblock {\em Differential Geometry and its Applications}, 29(2):179--206,
  March 2011.

\bibitem[Kam02]{kamada_braid_2002}
Seiichi Kamada.
\newblock Braid and {Knot} {Theory} in {Dimension} {Four}.
\newblock volume~95 of {\em Mathematical {Surveys} and {Monographs}},
  Providence, Rhode Island, May 2002. American Mathematical Society.

\bibitem[KKV24]{KKV}
Rafal Komendarczyk, Robin Koytcheff, and Ismar Volić.
\newblock Diagrams for primitive cycles in spaces of pure braids and string
  links.
\newblock {\em Annales de l'Institut Fourier}, 74(4):1745--1807, 2024.

\bibitem[Koh16]{kohno2016higher}
Toshitake Kohno.
\newblock Higher holonomy of formal homology connections and braid cobordisms.
\newblock {\em Journal of Knot Theory and Its Ramifications}, 25(12):1642007,
  2016.

\bibitem[Koh20]{kohno2020higher}
Toshitake Kohno.
\newblock Higher holonomy maps for hyperplane arrangements.
\newblock {\em European Journal of Mathematics}, 6(3):905--927, 2020.

\bibitem[Koh21]{kohno2021higher}
Toshitake Kohno.
\newblock Higher holonomy and iterated integrals.
\newblock In A.~Papadopoulos, editor, {\em Topology and Geometry, A collection
  of papers dedicated to Vladimir G. Turaev}, pages 307--325. European
  Mathematical Society Press, 2021.

\bibitem[Koh22]{Koh2022}
Toshitake Kohno.
\newblock Formal connections, higher holonomy functors and iterated integrals.
\newblock {\em Topology and its Applications}, 313:107985, 2022.

\bibitem[LV14]{lambrechtsvolic}
Pascal Lambrechts and Ismar Voli{\'c}.
\newblock Formality of the little $ n $-disks operad.
\newblock {\em Memoirs of the American Mathematical Society}, 230(1079), 2014.

\bibitem[Oht02]{ohtsuki2002quantum}
Tomotada Ohtsuki.
\newblock {\em Quantum invariants: A study of knots, 3-manifolds, and their
  sets}, volume~29.
\newblock World Scientific, 2002.

\bibitem[SW07]{schreiber_parallel_2007}
Urs Schreiber and Konrad Waldorf.
\newblock Parallel transport and functors.
\newblock {\em J. Homotopy Relat. Struct.}, 4, 05 2007.

\bibitem[SW08]{schreiber_smooth_2008}
Urs Schreiber and Konrad Waldorf.
\newblock Smooth functors vs. differential forms.
\newblock {\em Homology, Homotopy Appl.}, 13, 02 2008.

\end{thebibliography}

\end{document}